\documentclass[11pt]{article}
\usepackage{amsmath,amsthm}
\usepackage{amssymb}
\usepackage{bm}
\usepackage{subcaption}
\usepackage{xcolor}
\usepackage{tikz}
\usetikzlibrary{calc}

\usepackage{geometry}
\usepackage[colorlinks=true,
linkcolor=blue,citecolor=blue,
urlcolor=blue,pagebackref]{hyperref}

\makeatletter
\def\@seccntDot{.}
\def\@seccntformat#1{\csname the#1\endcsname\@seccntDot\hskip 0.5em}
\renewcommand\section{\@startsection{section}{1}{\z@}%
{18\p@ \@plus 6\p@ \@minus 3\p@}%
{9\p@ \@plus 6\p@ \@minus 3\p@}%
{\large\bfseries\boldmath}}
\renewcommand\subsection{\@startsection{subsection}{2}{\z@}%
{15\p@ \@plus 6\p@ \@minus 3\p@}%
{6\p@ \@plus 6\p@ \@minus 3\p@}%
{\itshape}}
\renewcommand\subsubsection{\@startsection{subsubsection}{3}{\z@}%
{12\p@ \@plus 6\p@ \@minus 3\p@}%
{\p@}%
{}}
\makeatother

\usepackage{microtype}

\theoremstyle{plain}
\newtheorem{theorem}{Theorem}[section]
\newtheorem{lemma}{Lemma}[section]
\newtheorem{corollary}{Corollary}[section]
\newtheorem{proposition}{Proposition}[section]

\newtheorem{conjecture}{Conjecture}[section]

\theoremstyle{definition}

\numberwithin{equation}{section}
\allowdisplaybreaks
\title{A complete solution to the Boots-Royle/Cao-Vince conjecture}

\author{Lele Liu\footnote{School of Mathematical Sciences, Anhui University, Hefei 230601,
P.R. China. E-mail: \texttt{liu@ahu.edu.cn} (L. Liu). Supported by the National
Natural Science Foundation of China (No. 12471320), and Anhui Provincial Natural Science Foundation for Excellent Young Scholars (No. 2408085Y003).}\,~~~~
Bo Ning\footnote{College of Cryptology and Cyber Science \& College of Computer Science, Nankai University, Tianjin 300350, P.R. China.
E-mail: \texttt{bo.ning@nankai.edu.cn} (B. Ning). Supported by the National Natural Science
Foundation of China (No. 12371350) and Fundamental Research Funds for the Central Universities, Nankai University (No. 63243151).}\,~~~~
Yi Wang\footnote{School of Mathematical Sciences, Anhui University, Hefei 230601,
P.R. China. Corresponding author, E-mail: \texttt{wangy@ahu.edu.cn} (Y. Wang).
Supported by the National Natural Science Foundation of China (Nos. 12571360 and 12331012).}
}

\date{}

\begin{document}
\maketitle

\begin{abstract}
Boots and Royle, and independently Cao and Vince, conjectured that the join of an edge and a path 
on $n-2$ vertices is the unique planar graph of maximum adjacency spectral radius for $n\geq 9$. 
Tait and Tobin (J. Combin. Theory Ser. B, 2017) proved the conjecture for sufficiently large order. In this paper, 
we completely resolve the Boots-Royle/Cao-Vince conjecture.
\par\vspace{2mm}
\par\vspace{2mm}
\end{abstract}

\section{Introduction}
Let $G$ be a simple graph, and let $\lambda(G)$ denote the spectral radius of its
adjacency matrix. A classical problem lying at the intersection of spectral graph theory and topological graph theory is to determine the maximum possible value of \(\lambda(G)\) over all planar graphs $G$ of a fixed order, and to characterize all planar graphs that attain this maximum. The study of this problem goes back at least to Schwenk and
Wilson~\cite{SchwenkWilson1978}. In 1988, Hong \cite{Hong1988} proved that every planar
graph of order $n$ satisfies $\lambda(G)\leq \sqrt{5n-11}$ by using his spectral 
bound and Euler's formula. Cao and Vince~\cite{CaoVince1993} subsequently improved this
estimate to $\lambda(G) < 4 + \sqrt{3n-9}$. Further improvements were obtained by 
Hong~\cite{Hong1995,Hong1998}, and Ellingham and Zha~\cite{EllinghamZha2000} later established the bound
$\lambda(G)\leq 2+\sqrt{2n-6}$. These results progressively narrowed the gap between the known upper bounds
and the natural planar constructions $K_2\vee P_{n-2}$, where $\vee$ denotes the graph join. They also shifted attention from merely
estimating the spectral radius to the more delicate problem of determining the exact extremal graph.

Two related spectral extremal conjectures concern outerplanar and planar graphs.
In their 1990 survey, Cvetkovi\'c and Rowlinson \cite{CvetkovicRowlinson1990} conjectured
that $K_1\vee P_{n-1}$ attains the maximum spectral radius among all outerplanar graphs of order $n$. Following the partial results of Cvetkovi\'c and Rowlinson \cite{CvetkovicRowlinson1990} and of
Cao and Vince \cite{CaoVince1993}, Tait and Tobin~\cite{TaitTobin2017} confirmed the
conjecture for all sufficiently large $n$. The remaining finite part was
eventually settled by Lin and Ning~\cite{LinNing2021}, who proved that the
conjectured fan is extremal for every $n\geq 2$, with $n=6$ as the unique
exception. The proof uses a result of Shu and Hong \cite{ShuHong2000}, which states that $\lambda(G)\leq 3/2 + \sqrt{n - 7/4}$.
Thus, the outerplanar spectral extremal problem is now completely
understood.

The second and more difficult conjecture is about planar graphs. Motivated in part by the use 
of the spectral radius as a measure of the connectivity of planar networks, Boots and Royle~\cite{BootsRoyle1991}
conjectured in 1991 that, for every $n\geq 9$, the unique planar graph of order
$n$ with maximum spectral radius is $K_2\vee P_{n-2}$. The same conjecture was proposed independently 
by Cao and Vince~\cite{CaoVince1993}. 

\begin{conjecture}[Boots-Royle-Cao-Vince]
Among all planar graphs on $n\geq 9$ vertices, $K_2\vee P_{n-2}$ attains the maximum spectral radius.   
\end{conjecture}

\begin{figure}[htbp]
\centering
\begin{tikzpicture}[main_node/.style={circle,draw,fill=blue!40,inner sep=.65mm}, scale=.9]
\node[main_node] (u1) at (-0.6, 1.3) {};
\node[main_node] (u2) at (-0.6, -1.3) {};
\node[main_node] (v4) at (5, 0) {};

\foreach \i in {0,1,...,3}
{%
\node[main_node] (v\i) at (\i, 0) {};
\draw[thick] (u1) -- (v\i);
\draw[thick] (u2) -- (v\i);
}

\draw[thick][thick] (v0) -- (v1) -- (v2) -- (v3);
\draw[thick] (u1) -- (u2);
\draw[thick] (u1) -- (v4) -- (u2);
\draw[thick] (v3) -- (3.2,0);
\node at (4,0) {$\cdots$};

\foreach \i in {0,1,...,3}
{%
\node[main_node] (v\i) at (\i, 0) {};
}

\node[left=2pt] at (u1) {$u_1$};
\node[left=2pt] at (u2) {$u_2$};
\node[anchor=north west] at (v0) {\small $v_1$};
\node[below=1pt] at (v1) {\small $v_2$};
\node[below=1pt] at (v4) {\small $v_{n-2}$};
\end{tikzpicture}
\caption{The graph $K_2\vee P_{n-2}$}
\label{fig:candidate}
\end{figure}
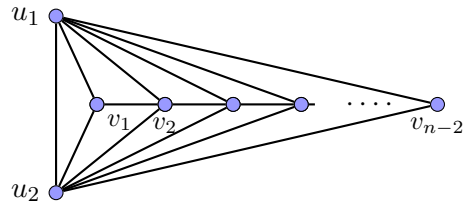

An important improvement was announced in Guiduli's 1996 doctoral dissertation, \emph{Spectral Extrema for Graphs}
, at the University of Chicago, written under the supervision of Babai~\cite{Guiduli1996}. It was later reported 
that Guiduli and Hayes had proved, in an unpublished preprint, that the planar conjecture holds for
sufficiently large $n$; however, that manuscript never appeared in published form (see also the introduction 
of Tait and Tobin \cite{TaitTobin2017}). In 2017, Tait and Tobin~\cite{TaitTobin2017} proved
that $K_2\vee P_{n-2}$ is indeed the unique planar graph of maximum spectral radius whenever $n$ is sufficiently large.

The original conjecture was formulated for $n\geq 9$. The purpose of this paper is to provide a complete proof of the conjecture. 

\begin{theorem}\label{thm:main}
For every $n\geq 3$, except for $n\in\{7,8\}$, the graph $K_2\vee P_{n-2}$ is the unique 
planar graph on $n$ vertices with a maximum adjacency spectral radius. For $n\in\{7,8\}$, 
the unique extremal graphs are $E_7$ and $E_8$, respectively, shown in Figure~\ref{fig:extremal-graphs-7-8}.
\end{theorem}

\begin{figure}[htbp]
\centering
\subcaptionbox{$E_7$ \label{subfig:7-vertex}}[.44\textwidth]{%
\begin{tikzpicture}[main_node/.style={circle,draw,fill=blue!40,inner sep=.65mm}, scale=.9]

\node[main_node] (0) at (1.8, 0) {};
\node[main_node] (1) at (5, 0) {};
\node[main_node] (2) at (3.2, -.3) {};
\node[main_node] (3) at (0, -1.5) {};
\node[main_node] (4) at (0.8, 0) {};
\node[main_node] (5) at (0, 1.5) {};
\node[main_node] (6) at (3.2, .3) {};

\draw[thick] (5) -- (0) -- (1) -- (2) -- (3) -- (4) -- (5) -- (6) -- (0) -- (3) -- (5) -- (1) -- (3);
\draw[thick] (2) -- (0) -- (4); 
\draw[thick] (1) -- (6); 
\end{tikzpicture}
}
\subcaptionbox{$E_8$ \label{subfig:8-vertex}}[.44\textwidth]{%
\begin{tikzpicture}[main_node/.style={circle,draw,fill=blue!40,inner sep=.65mm}, scale=.9]

\node[main_node] (0) at (0, -1.5) {};
\node[main_node] (1) at (0, 1.5) {};
\node[main_node] (2) at (.6, 0) {};
\node[main_node] (3) at (1.5, 0) {};
\node[main_node] (4) at (2.5, -.3) {};
\node[main_node] (5) at (4, 0) {};
\node[main_node] (6) at (5.2, 0) {};
\node[main_node] (7) at (2.5, .3) {};

\draw[thick] (0) -- (1) -- (2) -- (3) -- (4) -- (5) -- (6) -- (0) -- (2);
\draw[thick] (0) -- (3) -- (1) -- (5) -- (3) -- (7) -- (5) -- (0) -- (4);
\draw[thick] (6) -- (1) -- (7);
\end{tikzpicture}
}
\caption{The planar graphs with maximum spectral radius for $n=7,8$}
\label{fig:extremal-graphs-7-8}
\end{figure}
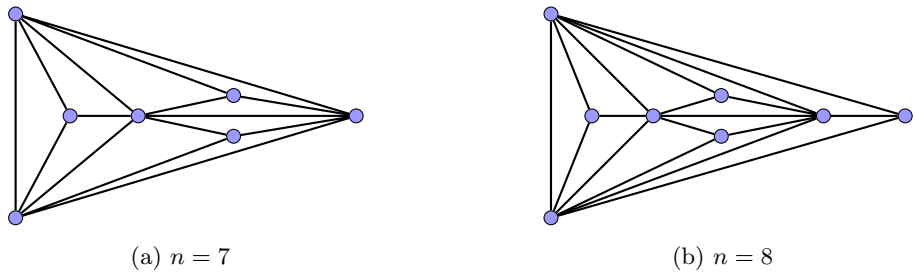

The corresponding problem for graphs embeddable on surfaces of arbitrary genus provides a broader 
setting for these results. Let $\operatorname{spex}(n,\gamma)$ denote the maximum spectral radius of an
$n$-vertex graph embeddable on a surface of Euler genus $\gamma$. Hong \cite{Hong1995, Hong1998}
initiated its systematic study and obtained, among other genus-dependent bounds,
$\operatorname{spex}(n,\gamma) \leq 1+\sqrt{3n+6\gamma-8}$. Ellingham and Zha \cite{EllinghamZha2000} improved this to $2+\sqrt{2n+8\gamma-6}$. Very recently, 
Zhai, Fang, and Lin~\cite{ZhaiFangLin2026} proved that, for each fixed $\gamma$ and all
sufficiently large $n$, every spectral extremal graph is obtained from $K_2\vee P_{n-2}$ by 
adding exactly $3\gamma$ edges inside a bounded core. Related extremal questions have also 
been studied for the tensor spectral radius of uniform hypergraphs. In particular, Ellingham, Lu, and Wang~\cite{EllinghamLuWang2022} established an outerplanar $3$-uniform hypergraph analog of 
the Cvetkovi\'c--Rowlinson conjecture for all sufficiently large orders.

We conclude this section with an outline of the paper. In Section \ref{sec:notation}, we
introduce the notation and preliminary tools, including a lower bound on the extremal spectral radius
supplied by $K_2\vee P_{n-2}$. Section \ref{sec:series-lemmas} focuses on cubic vertices in
an extremal plane triangulation, deriving constraints on the three largest Perron entries that lead to the structural inequalities essential for the subsequent arguments. In
Section \ref{sec:first-universal-vertex}, we prove that every extremal graph must contain a vertex of degree $n-1$. Section \ref{sec:second-universal-vertex} adapts the analysis from Section \ref{sec:series-lemmas}, now in the setting of outerplanar graphs, to force the existence of a second universal vertex. Finally, Section \ref{sec:finish-proof}
uses the two universal vertices and planarity to identify the remaining graph
as a path, which completes the proof.

\section{Notation and preliminary results}
\label{sec:notation}

For a graph $G$ on $n$ vertices, we write $e(G)$ for the number of edges of $G$. For a vertex $v$, 
we write $N_G(v)$ and $N_G[v]$ for its open and closed neighborhoods, respectively, 
and $d_G(v)$ for its degree. In the above notation, we omit the subscript $G$ 
when $G$ is clear from context. 
We refer to a vertex of degree $3$ as a \emph{cubic vertex}, and a vertex of degree $n-1$ as a \emph{universal vertex}.
We also use the usual facts about plane triangulations and maximal outerplanar graphs; 
see, for example, \cite[Chapter 10]{BondyMurty}.

\begin{theorem}[Liu--Weng \cite{LiuWeng}]\label{thm:degree}
Let $G$ be a connected graph on $n$ vertices with degree sequence
$d_1\geq d_2\geq\cdots\geq d_n$. For every $1\leq j\leq n$,
\[
\lambda(G)\leq \frac{d_j - 1 + \sqrt{(d_j+1)^2 + 4\sum_{i=1}^{j-1} (d_i - d_j)}}{2}.
\]
\end{theorem}

The following lemma is a consequence of Cauchy's interlacing theorem, also known as the Poincar\'e 
Separation Theorem (see, e.g., \cite{Haemers1995} or \cite[Corollary 4.3.37]{Horn-Johnson2012}).

\begin{lemma}[\cite{Haemers1995}; {\cite[Corollary 4.3.37]{Horn-Johnson2012}}]\label{lem:interlace-inequality}
Let $S$ be a real $n\times m$ matrix such that $S^{\mathrm{T}} S = I_m$ and let $A$ be a 
symmetric $n\times n$ matrix with eigenvalues $\lambda_1\geq\cdots\geq\lambda_n$. Define $B = S^{\mathrm{T}} AS$ 
and let $B$ have eigenvalues $\mu_1\geq\cdots\geq\mu_m$. Then 
\[
\lambda_i(A) \geq\mu_i(B)\geq\lambda_{i+n-m}(A), \quad i=1,\ldots,m.
\]
\end{lemma}

We next derive a lower bound on $\lambda(K_2\vee P_{n-2})$. Put
\[
 H_n:= K_2\vee P_{n-2}, \qquad r_n:= \frac{3 + \sqrt{8n-19}}{2}.
\]
It is clear that $r_n^2 - 3r_n = 2n-7$, which will be used repeatedly later.
\begin{lemma}\label{lem:benchmark}
For every $n\geq 12$, $\lambda (H_n) > r_n$.
\end{lemma}

\begin{proof}
Let $u_1,u_2$ be the two universal vertices of $H_n$, and let
$v_1,\ldots,v_{n-2}$ be its path vertices in their natural order.
Consider the orthonormal vectors
\[
 f_1=\frac{\bm{e}_{u_1}+\bm{e}_{u_2}}{\sqrt{2}},
 \qquad
 f_2=\frac{1}{\sqrt{n-2}}\sum_{i=1}^{n-2}\bm{e}_{v_i}.
\]
For $S=(f_1\ f_2)$, we have $S^{\top}S=I_2$, and
\[
 Q_n=S^{\top}A(H_n)S
 =\begin{bmatrix}
 1 & \sqrt{2(n-2)}\\
 \sqrt{2(n-2)} & 2-\dfrac{2}{n-2}
 \end{bmatrix}.
\]
By Lemma \ref{lem:interlace-inequality},
$\lambda(H_n)\geq\rho(Q_n)$, where $\rho(Q_n)$ is the largest
eigenvalue of $Q_n$.
Using $r_n^2-3r_n=2n-7$, we obtain
\begin{align*}
 \det(r_nI_2-Q_n)
 &=(r_n-1)\left(r_n-2+\frac{2}{n-2}\right)-2(n-2)\\
 &=-1+\frac{2(r_n-1)}{n-2}
 =\frac{2r_n-n}{n-2}.
\end{align*}
For $n\geq12$,
$2r_n=3+\sqrt{8n-19}<n$, and so $\det(r_nI_2-Q_n)<0$.
Since $Q_n$ is real and symmetric, this means that $r_n$ lies strictly
between its two eigenvalues. Hence
$\lambda(H_n)\geq\rho(Q_n)>r_n$, as required.
\end{proof}

\section{Cubic vertices in an extremal triangulation}
\label{sec:series-lemmas}

Throughout the remainder of the paper, let $G$ be a planar graph of order $n\geq 15$ having 
maximum spectral radius, and put $\lambda := \lambda(G)$. Let $\bm{x} = (x_v)_{v\in V(G)}$ 
be a positive Perron vector, normalized by $\|\bm{x}\|_1 = 1$. Since $H_n$ is planar, extremality and Lemma~\ref{lem:benchmark} imply $\lambda > r_n$.

Two consequences of $G$ being planar that we will use frequently are that $G$ has at most $3n-6$ edges and $G$ does not contain $K_{3,3}$ as a subgraph.
Since every planar graph on the same vertex set is contained in a maximal planar graph, 
it follows that $G$ is itself connected and edge-maximal. We therefore take a plane triangulation of 
$G$, which gives $e(G) = 3n - 6$.

Cubic vertices are essential to the extremal structure; accordingly, we establish 
some results concerning such vertices.

\begin{lemma}\label{lem:twoears}
The graph $G$ has at least two cubic vertices.
\end{lemma}

\begin{proof}

Assume that $G$ contains at most one cubic vertex.  Let $d=d_{n-1}$. Clearly, $d\geq 4$; 
moreover $d\leq 5$, since otherwise $\sum_vd(v)\geq 3 + 6(n-1) > 6n - 12$. Applying Theorem \ref{thm:degree} 
with $j=n-1$ gives 
\begin{align*}
2\lambda & \leq d-1 +\sqrt{(d+1)^2 + 4\sum_{i=1}^{n-2} (d_i - d)} \\
& \leq d-1 + \sqrt{(d-1)^2 + 4[6n-15-(n-2)d]}.
\end{align*}
Hence, $\lambda$ is bounded above by the positive root of
\begin{equation}\label{eq:f_dX}
 f_d(x) = x^2 - (d - 1) x - [6n - 15 - (n-2) d].
\end{equation}
Note that the quantity $6n - 15 - (n-2) d$ in \eqref{eq:f_dX} is positive for $d=4,5$ and $n\geq 15$, 
so $f_d(x)$ has one negative and one positive root. Using the identity $r_n^2-3r_n=2n-7$, we obtain
$f_d(r_n) = (d-4)(n-2-r_n) \geq  0$, since $d \geq 4$ and $r_n<n-2$. Thus $\lambda \leq  r_n$, a contradiction.
\end{proof}

For convenience, we introduce the following notation throughout the remainder of the paper:
\[
\mathcal{S} = \{v\in V(G): d(v) = 3\}, \qquad k = |\mathcal{S}|.
\]
In a plane triangulation, the three neighbors of any cubic vertex $v$ form a triangle, 
which we refer to as the \emph{support triangle} of $v$. The following lemma captures the 
elementary planar configuration arising from such vertices.

\begin{lemma}\label{lem:cubic-core}
Let $T$ be any plane triangulation of $G$. Then the following conclusions hold:
\begin{enumerate}
\setlength{\itemsep}{-1pt}
\item[$(1)$] $\mathcal{S}$ is an independent set in $G$, and vertices in $\mathcal{S}$ have pairwise distinct support
triangles. 

\item[$(2)$] $T\setminus\mathcal{S}$ is still a triangulation on $n - k$ vertices, and
$2\leq k\leq\lfloor (2n-4)/3 \rfloor$.
\end{enumerate}
\end{lemma}

\begin{proof}
(1) Since the neighbors of a cubic vertex in $T$ form a triangle, if two cubic vertices 
$u,v$ were adjacent, then the two facial triangles incident with $uv$ would force four
vertices to induce $K_4$. Every face of that $K_4$ meets $u$ or
$v$, so inserting any further vertex into the triangulation would
increase the degree of at least one of them. This would force $n=4$, contrary to $n\geq 15$.

Similarly, if two nonadjacent cubic vertices $u,v$ had the same support triangle, the 
five vertices would form $K_5-uv$, with the two
cubic vertices on opposite sides of the common triangle. Every face of this
triangulation meets one of $u,v$, so no further vertex could be added
without increasing the degree of $u$ or $v$. This would force $n=5$, contrary to $n\geq 15$.

(2) When a cubic vertex is deleted from a triangulation, the three incident facial triangles of the vertex are replaced with its support triangle. All other cubic vertices remain of degree three in the resulting graph. Hence, we may delete $k$ cubic vertices successively, and the outcome is also a triangulation $R$. By Euler's formula, we have $|\mathcal{F}(R)| = 2(n - k) - 4.$
By (1), the $k$ support triangles are distinct faces of $R$, so $k \leq 2(n - k) - 4$,
which is equivalent to the asserted general upper bound. 
\end{proof}

Lemma \ref{lem:cubic-core} shows that  removing all cubic vertices from a plane triangulation 
of $G$ leaves a graph $R$ that is still a triangulation. We refer to this remaining graph $R$ as 
the \emph{core} of $G$. The face of $R$ that arises from the deletion of the cubic vertex $u$ is called \emph{the occupied face} corresponding to $u$. Every unoccupied face of $R$ is also a face of $G$. Each occupied face is the support
 triangle of the corresponding cubic vertex and is replaced in $G$ by the three faces
incident with that vertex.


To translate the structure of cubic vertices into spectral information, we define
\[
c_v:= \max\{0,\, d(v) - 4\}, \quad E:= \sum_{v\in V(G)} c_vx_v, \quad 
L := \sum_{v\in\mathcal{S}} x_v, \quad r_v:= \sum_{w\notin N[v]} x_w.
\]

\begin{lemma}\label{lem:identities}
With the notation above, we have
\begin{equation}\label{eq:lambda-E-L}
\lambda = 4 + E - L \quad \text{and} \quad
\sum_{v\in V(G)} c_v = 2n - 12 + k.
\end{equation}
Moreover, for every vertex $v$, $r_v = 1 - (\lambda + 1) x_v$.
\end{lemma}

\begin{proof}
For each vertex $v$, the eigenvalue--eigenvector equation at $v$ gives $\lambda x_v = \sum_{u\in N(v)} x_u$. 
Summing over all vertices and using $\|\bm{x}\|_1 = 1$ gives
\begin{equation}\label{eq:lambda-sum-degree}
 \lambda = \sum_{v\in V(G)} \sum_{u\in N(v)} x_u = \sum_{v\in V(G)} d(v)x_v.
\end{equation}
Since every non-cubic vertex has degree at least four and contributes
$(4 + c_v) x_v$ to the right-hand side of \eqref{eq:lambda-sum-degree}, while every cubic vertex 
contributes $3x_v=(4-1)x_v$, it follows from $\|\bm{x}\|_1 = 1$ that
\[
 \lambda = \sum_{v\in V(G)} d(v)x_v = \sum_{v\in \mathcal{S}} (4-1)x_v + \sum_{v\in V(G)\setminus\mathcal{S}} (4 + c_v) x_v
 = 4 - L + E.
\]
This proves the first identity in \eqref{eq:lambda-E-L}.

Since $e(G) = 3n - 6$, Euler's formula gives $\sum_v (d(v)-4) = 2e(G) - 4n = 2n - 12$. On the other hand,
\[
\sum_{v\in V(G)} (d(v)-4) = \sum_{v\notin \mathcal{S}} (d(v) - 4) + \sum_{v\in\mathcal{S}} (d(v) - 4) 
= \sum_{v\in V(G)} c_v - k, 
\]
proving the second identity in \eqref{eq:lambda-E-L}.

Finally,  $\|\bm{x}\|_1 = 1$ and the eigenvalue–eigenvector equation at $v$ give
\[
r_v = 1 - x_v - \sum_{w\in N(v)} x_w = 1 - (\lambda + 1) x_v,
\]
as desired. 
\end{proof}

Relabel the vertices of $G$ as $v_1,v_2,\ldots,v_n$, write $x_i = x_{v_i}$, and
assume henceforth that $x_1\geq x_2\geq\cdots\geq x_n>0$.
Let $x^*:= \max_{v\in\mathcal{S}} x_v$. For a face $F$ of $G$, let $S(F):=\sum_{v\in F}x_v$. For the support triangle $T$ of a cubic vertex, we also use $S(T)$ to denote $\sum_{v\in T}x_v$.  

\begin{lemma}\label{lem:cubic-x-lower-bound}
For any $v\in\mathcal{S}$ and any face $F$ of $G-v$, we have 
$\sum_{u\in F} x_u \leq\lambda x_v$.
\end{lemma}

\begin{proof}
Let $T$ be the  support triangle $T$ of $v$. Place $v$ in the interior of $F$, join it to the
three vertices of $F$, and denote the resulting triangulation by $G_F$. Let
$S(F) := \sum_{u\in F} x_u$. The Rayleigh principle implies that
\begin{align*}
\lambda (G_F) - \lambda
& \geq \frac{\bm{x}^{\top} A(G_F) \bm{x}}{\bm{x}^{\top} \bm{x}} - \frac{\bm{x}^{\top} A(G) \bm{x}}{\bm{x}^{\top} \bm{x}} \\
& = \frac{2x_v}{\bm{x}^{\top} \bm{x}} (S(F) - S(T))
= \frac{2x_v}{\bm{x}^{\top} \bm{x}} (S(F) - \lambda x_v).
\end{align*}
Since $G$ is extremal, $\lambda(G_F)\leq\lambda$. Hence
$S(F)\leq \lambda x_v$.
\end{proof}

Lemma \ref{lem:cubic-x-lower-bound} allows us to refine our understanding of how the Perron 
entries are distributed. With this in mind, for $t>0$, we define  the strictly increasing function as
$$L_0(t):= \frac{k}{2n -  4 - (1-\frac{3}{t})k}$$ and put
$\ell:= \frac{2}{3} + \frac{1}{\lambda}.$

\begin{lemma}\label{lem:relocation}
For every $v\in\mathcal{S}$, we have
\begin{equation}\label{eq:x-upper-bound}
\ell x^* \leq x_v \leq\frac{3}{\lambda(\lambda + 1)}.
\end{equation}
Moreover, the maximum entry $x^*$ satisfies
\begin{equation}\label{eq:M-upper-bound}
\frac{L}{k} \leq x^* \leq\frac{L}{1+(k-1)\ell}.
\end{equation}
\end{lemma}

\begin{proof}
We first show that $x_v\geq \ell x^*$ for every $v\in\mathcal{S}$. By Lemma \ref{lem:twoears}, we have $k\geq 2$.  Choose $u\in \mathcal{S}$ with $x_u=x^*$, and let $v \in \mathcal{S}\setminus \{u\}$ be arbitrary.
Write $N(u) = \{a,b,c\}$. By Lemma \ref{lem:cubic-core}, $u$ remains a cubic vertex of $G-v$, and
$uab,ubc,uca$ are faces of $G-v$. Moving $v$ inside the interior of the face among these 
three that maximizes the sum of the Perron entries, Lemma \ref{lem:cubic-x-lower-bound} yields
\begin{align*}
\lambda\, x_v
& \geq x_u + \max\{x_a + x_b, x_b + x_c, x_c + x_a\} \\
& \geq x_u + \frac{2}{3} (x_a + x_b + x_c)
= \Big(1 + \frac{2\lambda}{3}\Big) x_u.
\end{align*}
Thus $x_v\geq \ell x_u = \ell x^*$. Moreover, since $\lambda > r_n > 5$, we have $\ell < 1$, and 
therefore $x_v\geq \ell x^*$ for all cubic vertices. 

For the upper bound in \eqref{eq:x-upper-bound}, it follows from Lemma \ref{lem:identities} that $r_w = 1 - (\lambda + 1) x_w \geq 0$ 
for each $w\in V(G)$, so $x_w\leq 1/(\lambda + 1)$. Now, using
the eigenvalue–eigenvector equation at a cubic vertex $v$, we get
\[
\lambda x_v = \sum_{w\in N(v)} x_w\leq\frac{3}{\lambda + 1},
\]
which proves the desired inequality.

Now we prove \eqref{eq:M-upper-bound}. The inequality $x^*\geq L/k$ follows because $x^*$ is the largest entry among all cubic
vertices. For the upper bound on $x^*$, we use the bound $x_v\geq \ell x^*$ for $v\in\mathcal{S}$, 
which implies $L\geq x^* + (k-1) \ell x^*$. 
\end{proof}

\begin{lemma}\label{lem:L-L0}
$L\geq L_0(\lambda)$.
\end{lemma}

\begin{proof}  Let $\mathcal{S}=\{u_1, \ldots, u_k\}$, and let $T_i$ be the support triangle of $u_i$ for each $1 \leq i \leq k$. If $F$ is
an unoccupied face of $R$, then $F$ is a face of $G-u_i$ for every
cubic vertex $u_i$. Let $S(F):=\sum_{u\in F}x_u$.
Relocating $u_i$ into $F$ and using Lemma \ref{lem:cubic-x-lower-bound} gives
$S(F)\leq S(T_i)$. 
Thus, every occupied face has a Perron sum at least that of every unoccupied
face. Consequently,the average Perron sum over the occupied faces is at least the average
over all faces of $R$.

On the one hand,
the average Perron sum of all occupied faces is
\[
\frac{1}{k}\sum_{i=1}^k S(T_i) = \frac{1}{k}\sum_{i=1}^k \lambda\, x_{u_i} =\frac{\lambda L}{k}.
\]
On the other hand, let $h_w$ be the number of support triangles
incident with a core vertex $w$. Then $d_R(w)=d_G(w)-h_w$, and
\begin{align*}
\sum_{F\in\mathcal{F}(R)} S(F)
& = \sum_{w\in V(R)} d_R(w) x_w
= \sum_{w\in V(R)} d_G(w)x_w - \sum_{w\in V(R)} h_wx_w \\
& = (\lambda - 3L) - \sum_{i=1}^k S(T_i) = \lambda - (\lambda + 3) L.
\end{align*}
Therefore, 
\[
\frac{\lambda L}{k} \geq \frac{\lambda - (\lambda + 3) L}{2(n - k) - 4}.
\]
Multiplying out and solving for $L$ gives $L\geq L_0(\lambda)$.
\end{proof}

The next three lemmas focus on the constraint relations among the three largest Perron entries. 
To simplify the presentation, we write $s:= x_1 + x_2$ and $z:= x_3$. For $i\geq 1$, define the sequence $\{\alpha_i\}$ by 
\[
\alpha_1 = \alpha_2 = 0, \quad 
\alpha_i = \frac{2i - 5}{2i + 1} \quad (i\geq 3).
\]

\begin{lemma}\label{lem:constraint}
$(\lambda + 1)s\leq 2 - \alpha_k L$.
\end{lemma}

\begin{proof}
First, applying Lemma \ref{lem:identities} to $v_1$ and $v_2$ gives
\[
r_{v_1} + r_{v_2} = \big(1 - (\lambda + 1) x_1\big) + \big(1 - (\lambda + 1) x_2\big) 
= 2 - (\lambda + 1) s.
\]
Therefore it suffices to show that $r_{v_1} + r_{v_2} \geq \alpha_k L$.
Consider the contribution of the cubic vertices to $r_{v_1} + r_{v_2}$.
By the definition of $r_v$, a cubic vertex contributes nothing to $r_{v_1}+r_{v_2}$ exactly when it belongs to
$N[v_1]\cap N[v_2]$. If $v_1v_2\notin E(G)$, no cubic vertex can lie in this
intersection because the neighbors of any cubic vertex form a triangle. If
$v_1v_2\in E(G)$, then every cubic vertex distinct from $v_1,v_2$
that belongs to $N[v_1]\cap N[v_2]$ must be the third vertex of a face containing 
the edge $v_1v_2$. Since each edge of a plane triangulation is incident with exactly 
two faces, there are at most two such cubic vertices. The same conclusion holds if one of
$v_1,v_2$ is itself a cubic vertex, by independence of $\mathcal{S}$. Hence the sum of 
the Perron entries of cubic vertices that contribute nothing to  $r_{v_1} + r_{v_2}$ is 
at most $2x^*\leq 6L/(2k + 1)$, where we used \eqref{eq:M-upper-bound} and $\ell >2/3$. 
Every other cubic vertex is counted at least once in $r_{v_1} + r_{v_2}$, so
\[
r_{v_1} + r_{v_2}\geq L - \frac{6L}{2k+1} = \alpha_k L
\]
for $k\geq 3$. For $k=1,2$, we use the trivial lower bound zero.
\end{proof}

Let $U:=\{v_1,v_2,v_3\}$ denote the set of vertices corresponding to the three largest Perron entries.
For each \(w\in V(G)\setminus U\), let $m_w:= \big|\{v\in U: w\notin N[v]\}\big|$.

\begin{lemma}\label{lem:topthree}
$\lambda\, s + (\lambda - 2) z\leq 2$.
\end{lemma}

\begin{proof}
We first estimate $r_{v_1}+r_{v_2}+r_{v_3}$. Recall that for any vertex $v$, 
\[
r_v = 1 - x_v - \sum_{w\in N(v)} x_w = \sum_{w\notin N[v]} x_w.
\]
Hence, the contribution of $w\in V(G)\setminus U$ to $r_{v_1}+r_{v_2}+r_{v_3}$ is precisely $m_wx_w$.
There are at most two vertices in $V(G)\setminus U$ adjacent to all
three of $v_1,v_2,v_3$, since three such vertices together with $v_1,v_2,v_3$ would form a $K_{3,3}$, 
contrary to planarity. Consequently, apart from at most two exceptional vertices, every
$w\in V(G)\setminus U$ is nonadjacent to at least one of $v_1,v_2,v_3$, so $m_w\geq 1$. 
Since $w\notin U$ implies $x_w\leq x_3 = z$, the total sum of Perron entries of the exceptional 
vertices is at most $2z$. We obtain
\begin{align*}
r_{v_1} + r_{v_2} + r_{v_3}
& \geq \sum_{w\in V(G)\setminus U} m_wx_w
\geq \sum_{w\in V(G)\setminus U} x_w - 2z \\
& = 1 - x_1 - x_2 - x_3 - 2z = 1 - s - 3z.
\end{align*}

On the other hand, applying Lemma \ref{lem:identities} to $v_1,v_2,v_3$ gives
\begin{align*}
r_{v_1} + r_{v_2} + r_{v_3}
& = 3 - (\lambda + 1)(x_1 + x_2 + x_3)
= 3 - (\lambda + 1)(s + z).
\end{align*}
Combining this with the previous lower bound yields $3 - (\lambda + 1)(s + z)\geq 1 - s - 3z$, 
which, after rearrangement, becomes $\lambda s + (\lambda - 2)z\leq 2$.
\end{proof}

For every integer $k\geq2$, define
$\gamma_k:=2(k-6)/(3k)$.
The following lemma strengthens Lemma \ref{lem:topthree} when $k\geq7$.

\begin{lemma}\label{lem:weighted0}
If $k\geq7$, then
$\lambda s+(\lambda-2)z<2-\gamma_kL$.
\end{lemma}

\begin{proof}
Let $X:=\sum_{w\notin U}(m_w-1)x_w$.
We first prove
\begin{equation}\label{eq:temporary-ineq-1}
 \sum_{w\in\mathcal{S}\setminus U}(m_w-1)\geq k-6.
\end{equation}
For $m_w\in\{0,1,2,3\}$,
$m_w-1\geq1-\binom{3-m_w}{2}$.
The quantity $\binom{3-m_w}{2}$ counts pairs of vertices of $U$
in the support triangle of $w$.
Since $\mathcal{S}$ is independent, these pairs lie in
$U\setminus\mathcal{S}$.
Every such pair is an edge of the core $R$ and is contained in at
most two occupied faces of $R$.
Consequently,
\begin{align*}
 \sum_{w\in\mathcal{S}\setminus U}(m_w-1)
 &\geq k-|\mathcal{S}\cap U|
       -\sum_{w\in\mathcal{S}\setminus U}\binom{3-m_w}{2}\\
 &\geq k-|\mathcal{S}\cap U|
       -2\binom{3-|\mathcal{S}\cap U|}{2}
 \geq k-6,
\end{align*}
where the last inequality follows from
$|\mathcal{S}\cap U|\in\{0,1,2,3\}$.

There are at most two vertices outside $U$ with $m_w=0$, by the
same $K_{3,3}$ argument as in Lemma \ref{lem:topthree}.
Their total contribution to $X$ is at least $-2z$.
All other nonzero contributions are nonnegative. Since
$x_w\geq\ell x^*$ for $w\in\mathcal{S}$, we obtain
\[
X \geq\ell x^* \sum_{w\in\mathcal{S}\setminus U,\,m_w\geq 2} (m_w - 1) - 2z
\geq\ell x^*(k-6)-2z.
\]
Here the second inequality follows from
\eqref{eq:temporary-ineq-1}, since omitting its negative summands
can only increase the sum.
Now $k\geq7$, $\ell>2/3$, and $x^*\geq L/k$ imply
$X>\gamma_kL-2z$.
Therefore,
\begin{align*}
3-(\lambda+1)(s+z) 
& = r_{v_1}+r_{v_2}+r_{v_3} \geq\sum_{w\notin U} m_wx_w \\
& = 1-s-z+X >1-s-3z+\gamma_kL.
\end{align*}
Rearranging proves the lemma.
\end{proof}

Write $(\gamma_k)_+:=\max\{0,\gamma_k\}$.
For $2\leq k\leq6$, we have $(\gamma_k)_+=0$.
Thus Lemmas \ref{lem:topthree} and \ref{lem:weighted0} give the
following corollary.

\begin{corollary}\label{lem:refin}
$\lambda\, s+(\lambda-2)z\leq2-(\gamma_k)_+L$.
\end{corollary}

\section{Forcing a universal vertex}
\label{sec:first-universal-vertex}

The aim of this section is to prove that $G$ has a universal vertex. We will assume that no vertex is universal 
and derive a contradiction. A series of lemmas in Section \ref{sec:series-lemmas}
reduces the problem to a two-variable linear optimization in $s=x_1+x_2$ and $z=x_3$.

\begin{lemma}\label{lem:normalized-first}
For every integer $2\leq k\leq (2n-4)/3$, we have
$(n-6)s+kz-L_0(r_n)<r_n-4$.
\end{lemma}
\begin{proof}
Recall that $r_n>13/2$ and $r_n^2-3r_n=2n-7$.
The definition of $L_0$ and $k\leq(2n-4)/3$ give
\[
 \frac{k}{2n-4}<L_0(r_n)<\frac12.
\]

Set $C:=2-\gamma_k L_0(r_n)$, where
$\gamma_k=2(k-6)/(3k)$. Here, $\gamma_k$ is allowed to be negative.
Then $C>0$. By Lemma \ref{lem:L-L0}, Corollary \ref{lem:refin}, and the monotonicity of $L_0(t)$, we obtain 
\[
r_ns + (r_n-2)z\leq \lambda s + (\lambda-2)z   \leq  2 - (\gamma_k)_+ L_0(\lambda) \leq  2 - (\gamma_k)_+ L_0(r_n) \leq C.
\]
Therefore, $z\leq \frac{C - r_ns}{r_n - 2}$.
Moreover, Lemma~\ref{lem:identities} gives $x_i\leq1/(\lambda+1)$,
so $s\leq2/(r_n+1)$. Since $z\leq s/2$, we have
\begin{equation}\label{eq:z-upper-bound}
 0\leq z\leq\min\left\{\frac{s}{2},\frac{C-r_ns}{r_n-2}\right\}.
\end{equation}
In particular,
\begin{equation}\label{eq:temporary-ineq-2}
 (n-6)s+kz\leq
 \Big(n-6-\frac{kr_n}{r_n-2}\Big)s+\frac{kC}{r_n-2}.
\end{equation}
The two upper bounds for $z$ in \eqref{eq:z-upper-bound} coincide when
$\frac{s}{2}=\frac{C-r_ns}{r_n-2}$,
that is, $s=s_0:=\frac{2C}{3r_n-2}$.
For $0\leq s\leq s_0$, we have
\[
 (n-6)s+kz\leq\Big(n-6+\frac{k}{2}\Big)s
 \leq\Big(n-6+\frac{k}{2}\Big)s_0,
\]
since $n-6+k/2>0$. For $s\geq s_0$, we use the upper bound in
\eqref{eq:temporary-ineq-2}. Its value at $s=2/(r_n+1)$ is
\begin{align*}
B_1
 &=\Big (n-6-\frac{kr_n}{r_n-2}\Big)\frac{2}{r_n+1} +\frac{kC}{r_n-2} \\
 &=\frac{2(n-6)}{r_n+1}+\frac{k}{r_n-2}\Big(C-\frac{2r_n}{r_n+1}\Big).
\end{align*}
At $s=s_0$, the two upper bounds for $(n-6)s+kz$ coincide, with a common value
\[
 B_2=\Big(n-6+\frac{k}{2}\Big) s_0
 =\Big(n-6+\frac{k}{2}\Big)\frac{2C}{3r_n-2}
 =\frac{(2n-12+k)C}{3r_n-2}.
\]
If $s_0\leq 2/(r_n+1)$, the upper bound in
\eqref{eq:temporary-ineq-2} is at most $\max\{B_1,B_2\}$ on
$[s_0,2/(r_n+1)]$, since a linear function attains its maximum
at an endpoint. Together with the bound for $s\leq s_0$, this gives
\begin{equation} \label{eq:uniform}
  (n-6)s+kz\leq\max\{B_1,B_2\}.  
\end{equation}

If $s_0>2/(r_n+1)$, every feasible $s$ satisfies $s<s_0$,
and the bound for $s\leq s_0$ gives $(n-6)s+kz\leq B_2$.
Thus \eqref{eq:uniform} holds in both cases.

For the first endpoint, direct simplification gives
\begin{align*}
 B_1-L_0(r_n)-(r_n-4)
 &=\frac{2k-(r_n-2)}{(r_n-2)(r_n+1)}
   -L_0(r_n)\left(1+\frac{2(k-6)}{3(r_n-2)}\right)\nonumber\\
 &<\frac{Q(k)}{3(2n-4)(r_n-2)(r_n+1)},
\end{align*}
where $Q(k)=-2(r_n+1)k^2+3(r_n^2-r_n+12)k-3(2n-4)(r_n-2)$.
The strict inequality holds because
$1+2(k-6)/[3(r_n-2)]>0$ for $k\geq 2$.
The discriminant of $Q(k)$ is $-3H_4(r_n)$, where
$H_4(r_n)=5r_n^4-26r_n^3-43r_n^2+96r_n-480$.
Writing $u=r_n-13/2>0$, we obtain
\begin{equation*}
 H_4(r_n)=5u^4+104u^3+\frac{1435}{2}u^2+1734u+\frac{1797}{16}>0.
\end{equation*}
Since its leading coefficient is negative, $Q(k)<0$ for every real $k$.
Thus the first endpoint is strictly smaller than $L_0(r_n)+r_n-4$.

For the second endpoint, put $\theta=k/(2n-4)\in(0,1/3]$.
The definition of $L_0$ gives
\begin{equation}\label{eq:temporary-ineq-10}
 B_2-L_0(r_n)-(r_n-4)
 =\frac{Q_{r_n}(\theta)}{(3r_n-2)[3r_n-(3r_n-9)\theta]},
\end{equation}
where
\[
\begin{aligned}
 Q_{r_n}(\theta)={}&-2(2n-4)(4r_n-9)\theta^2
 +(7r_n^3-54r_n^2+172r_n-162)\theta\\
 &-3r_n^3+24r_n^2-42r_n-\frac{96r_n}{2n-4}.
\end{aligned}
\]
The denominator in \eqref{eq:temporary-ineq-10} is positive.
The quadratic $Q_{r_n}$ has a negative leading coefficient and discriminant
$-H_6(r_n)$, where
\[
 H_6(r_n)=47r_n^6-516r_n^5+988r_n^4+5652r_n^3
          -25720r_n^2+39744r_n-26244.
\]
For the same $u=r_n-13/2>0$,
\[
\begin{aligned}
 H_6(r_n)={}&47u^6+1317u^5+\frac{56017}{4}u^4
 +\frac{142955}{2}u^3\\
 &+\frac{2821697}{16}u^2+\frac{2778421}{16}u
 +\frac{1205279}{64}>0.
\end{aligned}
\]
Hence $Q_{r_n}(\theta)<0$. Both endpoints are therefore strictly
smaller than $L_0(r_n)+r_n-4$, and the desired inequality follows.
\end{proof}

\begin{proposition}\label{prop:universal}
The extremal graph $G$ has a universal vertex.
\end{proposition}

\begin{proof}
Assume, for a contradiction, that $d(v)\leq n-2$ for each $v\in V(G)$.
Then $c_v = \max\{0,\, d(v) - 4\} \leq n-6$ for every $v\in V(G)$.
By \eqref{eq:lambda-E-L}, we have $\sum_v c_v = 2n - 12 + k$. Since
$x_i\leq z$ for $i\geq 3$, it follows that
\begin{align}
E
& \leq z\sum_{v\in V(G)} c_v + (x_1 - z)c_{v_1} + (x_2 - z)c_{v_2} \nonumber \\
& \leq z(2n - 12 + k) + (n-6)(x_1 + x_2 - 2z) \nonumber \\
& = (n-6)s + kz. \label{eq:E-upper-bound}
\end{align}
Combining \eqref{eq:E-upper-bound} with Lemma \ref{lem:identities}, we get
$\lambda - 4 = E - L\leq (n - 6)s + kz - L$.
Note that $L > L_0(r_n)$. This yields
$\lambda - 4 < (n - 6)s + kz - L_0(r_n)$.
On the other hand, it follows from Lemma \ref{lem:normalized-first} that
$(n-6)s + kz - L_0(r_n) < r_n - 4< \lambda -4$,
a contradiction. This completes the proof.
\end{proof}

\section{Forcing a second universal vertex}
\label{sec:second-universal-vertex}

In Section \ref{sec:first-universal-vertex}, we proved that $G$ has a universal vertex, say $u$.
Recall that the Perron vector is normalized by
$\sum_{w\in V(G)}x_w=1$.
The eigenvalue equation at $u$ gives
\[
 \lambda x_u=\sum_{w\ne u}x_w=1-x_u,
 \quad\text{so}\quad x_u=\frac{1}{\lambda+1}.
\]
For every vertex $v\in V(G)$, we have
\[
 \lambda x_v=\sum_{w\in N_G(v)}x_w
 \leq\sum_{w\ne v}x_w=1-x_v.
\]
Hence $x_v\leq1/(\lambda+1)=x_u$ for every $v\in V(G)$.
Thus $u$ has the largest Perron entry, and we may take $u=v_1$.
Moreover, since all Perron entries are positive, equality holds
if and only if $v$ is a universal vertex. The aim of this section is to prove that $G$ has another universal vertex.

Fix a plane embedding of $G$. Since $G$ is maximal planar,
every face is a triangle. When $v_1$ and its incident edges are
deleted, all faces incident to $v_1$ merge into a single face whose boundary contains all vertices of $G-v_1$.
Thus, $G-v_1$ is outerplanar.
Moreover,
\[
 |E(G-v_1)|=3n-6-(n-1)=2(n-1)-3,
\]
 which implies that $G-v_1$ is maximal outerplanar.

Throughout this section, write $y := x_2$, $z := x_3$. The cubic vertices of $G$ are exactly 
the degree two vertices of the maximal outerplanar graph $G - v_1$. By Lemma~\ref{lem:cubic-core}, these
vertices form an independent set. So, they are nonconsecutive on the outer Hamilton cycle of $G-v_1$, which yields that $ k\leq\lfloor (n-1)/2\rfloor$.

\begin{lemma}\label{lem:outer-ear-mass}
$(\lambda + 1)y \leq 1 - \alpha_k L$.
\end{lemma}

\begin{proof}
Recall that $v_1$ is a universal vertex. Denote by $v_2$ the vertex whose Perron entry equals $y$.
Each cubic vertex of $G$ has degree two in the maximal outerplanar
graph $G-v_1$, so its neighbors there are precisely its two
neighbors on the outer Hamilton cycle.
Thus, if $v_2\notin\mathcal{S}$, every cubic neighbor of $v_2$
is one of its two outer-cycle neighbors. If $v_2\in\mathcal{S}$,
then $\mathcal{S}\cap N_G[v_2]=\{v_2\}$, since $\mathcal{S}$ is
independent. In either case, $N_G[v_2]$ contains at most two cubic
vertices, and hence
\[
 r_{v_2}\geq\sum_{u\in\mathcal{S}\setminus N_G[v_2]}x_u
 \geq L-2x^*.
\]
For $k\geq3$, Lemma \ref{lem:relocation} and $\ell>2/3$ give
\[
 2x^*\leq\frac{2L}{1+(k-1)\ell}
 \leq\frac{6L}{2k+1}.
\]
Therefore
\begin{equation*}
 r_{v_2}\geq L-\frac{6L}{2k+1}
 =\frac{2k-5}{2k+1}L=\alpha_kL.
\end{equation*}
For $k=2$, the same conclusion follows
from  $r_{v_2}\geq0$ and $\alpha_2=0$.
Finally, Lemma \ref{lem:identities} gives $r_{v_2}=1-(\lambda+1)y \geq \alpha_kL$, which proves the desired inequality.
\end{proof}

Substituting $x_1 = 1/(\lambda + 1)$ into Lemma \ref{lem:topthree}, we immediately obtain the following lemma.

\begin{lemma}\label{lem:lambda-y+lambda-z}
$\lambda\, y + (\lambda - 2) z
\leq \frac{\lambda + 2}{\lambda + 1}$. 
\end{lemma}

For $k\geq 7$, Lemma \ref{lem:weighted0} gives 
$\lambda s + (\lambda - 2) z < 2 - \frac{2(k - 6)}{3k} L$. Since $x_1 = 1/(1+\lambda)$ 
and $s = x_1 + y$, this can be rewritten as 
\[
\lambda\, y + (\lambda - 2) z < \frac{\lambda + 2}{\lambda + 1} - \frac{2(k-6)}{3k} L.
\]
 For every integer $k \geq 2$, let $\gamma_k':= \frac{2(k-4)}{3k}$. The next lemma shows that this estimate can be sharpened.

\begin{lemma}\label{lem:weighted1}
If $k\geq 5$, then
\[
\lambda\, y + (\lambda - 2)z
< \frac{\lambda + 2}{\lambda + 1} - \gamma_k' L.
\]
\end{lemma}

\begin{proof}
As in the proof of Lemma \ref{lem:weighted0}, set
$X:=\sum_{w\notin U}(m_w-1)x_w$.
By the same argument as in Lemma \ref{lem:outer-ear-mass},
each of $N_G[v_2]$ and $N_G[v_3]$ contains at most two cubic vertices.
Hence
\begin{equation}\label{eq:temporary-ineq-9}
 \bigl|\mathcal{S}\setminus\bigl(N_G[v_2]\cup N_G[v_3]\bigr)\bigr|
 \geq k-4.
\end{equation}
Every vertex counted in \eqref{eq:temporary-ineq-9} lies outside $U$
and has $m_w=2$, since $v_1$ is universal.
As in Lemma \ref{lem:weighted0}, at most two vertices have $m_w=0$,
and their total contribution to $X$ is at least $-2z$.
Thus Lemma \ref{lem:relocation} and $k\geq5$ give
\[
 X\geq\ell x^*(k-4)-2z
 >\frac{2(k-4)}{3k}L-2z=\gamma_k' L-2z.
\]
Hence
\[
\begin{aligned}
3-(\lambda+1)(s+z)
&=r_{v_1}+r_{v_2}+r_{v_3}\\
&\geq\sum_{w\notin U}m_wx_w
 =1-s-z+X\\
&>1-s-3z+\gamma_k'L.
\end{aligned}
\]
Rearranging gives
$\lambda s+(\lambda-2)z<2-\gamma_k'L$.
Substituting $s=1/(\lambda+1)+y$ proves the lemma.
\end{proof}

\begin{lemma}\label{lem:normalized-second}
We have
\begin{equation}\label{eq:second-univer}
\frac{n-5}{r_n+1} + (n-6)y + (k-1)z - L_0(r_n) < r_n - 4.
\end{equation}
\end{lemma}

\begin{proof} Since $k\leq (n-1)/2\leq (2n-4)/3$, the same argument as in the 
proof of Lemma \ref{lem:normalized-first} gives
\begin{equation}\label{eq:second-t-L0}
r_n > \frac{13}{2}, \qquad \frac{k}{2n-4} < L_0(r_n) < \frac{1}{2}.
\end{equation}
In view of Lemma \ref{lem:lambda-y+lambda-z}, Lemma \ref{lem:weighted1} and the monotonicity of \(L_0(t)\), we obtain
\[
z\leq\frac{1}{r_n-2} \left(\frac{r_n+2}{r_n+1} - (\gamma'_k)_+ L_0(r_n) -r_ny\right),
\]
where $(\gamma'_k)_+:= \max\{0,\,\gamma'_k\}$.
Together with Lemma \ref{lem:outer-ear-mass}, this yields 
\begin{equation}\label{eq:temporary-ineq-12}
(n-6)y + (k-1)z\leq \frac{k-1}{r_n-2}
\left(\frac{r_n+2}{r_n+1} - (\gamma'_k)_+ L_0(r_n)\right)
+ \left(n - 6 - \frac{(k-1)r_n}{r_n-2}\right)
\frac{1-\alpha_kL_0(r_n)}{r_n+1}.
\end{equation}
Here the coefficient $n-6-(k-1)r_n/(r_n-2)$ is positive. Indeed, since $2(k-1)\leq n-3$, we have
\begin{align*}
2(r_n-2)\left(n-6-\frac{(k-1)r_n}{r_n-2}\right)
& \geq 2(n-6)(r_n-2) - (n-3)r_n \\
& = (n-9)r_n - 4n + 24 > 0,
\end{align*}
where the last inequality follows from $n\geq 15$ and $r_n\geq 13/2$.
We now verify the desired inequality by considering the possible values
of $k$.
\smallskip

\noindent\emph{Case 1: $k=2$.}
Here $\alpha_2=(\gamma_2')_+=0$, so
\eqref{eq:temporary-ineq-12} gives
\[
\begin{aligned}
 (n-6)y+(k-1)z
 &\leq\frac{1}{r_n+1}\left(n-6-\frac{r_n}{r_n-2}\right)
       +\frac{r_n+2}{(r_n-2)(r_n+1)}\\
 &=\frac{n-6}{r_n+1}+\frac{2}{(r_n-2)(r_n+1)}.
\end{aligned}
\]
Also, the definition of $L_0$ gives
\[
 L_0(r_n)=\frac{2}{2n-6+6/r_n}
        =\frac{1}{n-3+3/r_n}.
\]
Using $2n-7=r_n^2-3r_n$ and $r_n>13/2$, we have
\[
 (r_n-2)(r_n+1)-2\left(n-3+\frac{3}{r_n}\right)
 =2r_n-3-\frac{6}{r_n}>0.
\]
Since both denominators are positive, this implies
$2/[(r_n-2)(r_n+1)]<L_0(r_n)$. Consequently,
\begin{align*}
 \frac{n-5}{r_n+1}+(n-6)y+(k-1)z-L_0(r_n)
 &\leq\frac{2n-11}{r_n+1}
       +\frac{2}{(r_n-2)(r_n+1)}-L_0(r_n)\\
 &<\frac{2n-11}{r_n+1}=r_n-4,
\end{align*}
where the last equality uses
$2n-11=r_n^2-3r_n-4=(r_n-4)(r_n+1)$.
This proves the desired inequality for $k=2$.
\smallskip

\noindent\emph{Case 2: $k=3$ or $k=4$}. By \eqref{eq:second-t-L0}, we have $L_0(r_n)>k/(2n-4)$. Substituting
this lower bound into \eqref{eq:temporary-ineq-12} gives, respectively,
\begin{align*}
\frac{n-5}{r_n+1}+(n-6)y+2z-L_0(r_n)-(r_n-4)
& < -\frac{3r_n^3-29r_n^2+117r_n-222}{14(r_n-2)(r_n+1)(2n-4)} < 0, \\
\frac{n-5}{r_n+1}+(n-6)y+3z-L_0(r_n)-(r_n-4)
& < -\frac{2(r_n^3-8r_n^2+16r_n-29)}{3(r_n-2)(r_n+1)(2n-4)} < 0.
\end{align*}
Hence, \eqref{eq:second-univer} holds for $k=3$ and $k=4$.
\smallskip

\noindent\emph{Case 3: $k\geq5$.}
Here $\alpha_k\geq5/11$ and $(\gamma_k')_+=\gamma_k'=2(k-4)/(3k)$.
The coefficient $n-6-(k-1)r_n/(r_n-2)$ is positive, as proved above.
Thus \eqref{eq:temporary-ineq-12} and $2n-11=(r_n-4)(r_n+1)$ give
\begin{align*}
& ~ \frac{n-5}{r_n+1}+(n-6)y+(k-1)z-L_0(r_n)-(r_n-4) \\
\leq & ~\frac{2(k-1)}{(r_n-2)(r_n+1)} - L_0(r_n)
\left[1+\frac{5}{11(r_n+1)}\left(n-6-\frac{(k-1)r_n}{r_n-2}\right)
 +\frac{2(k-1)(k-4)}{3k(r_n-2)}\right].
\end{align*}
It follows from  $L_0(r_n)>k/(2n-4)$ that
\begin{align}
& ~ \frac{n-5}{r_n+1}+(n-6)y+(k-1)z-L_0(r_n)-(r_n-4) \nonumber \\
< & ~\frac{2(k-1)}{(r_n-2)(r_n+1)} - \frac{k}{2n-4}
\left[1+\frac{5}{11(r_n+1)}\left(n-6-\frac{(k-1)r_n}{r_n-2}\right)
 +\frac{2(k-1)(k-4)}{3k(r_n-2)}\right]. \label{eq:lemma54-case3-bound}
\end{align}
Multiplying the right-hand side of the above inequality by  
$66(2n-4)(r_n-2)(r_n+1)$ and using $2n-4=r_n^2-3r_n+3$,
we obtain 
\begin{align}
& 132 (2n {-} 4)(k {-}1) {-} 66k(r_n {-} 2)(r_n {+}1) 
{-}30k\bigl[(n {-} 6)(r_n {-} 2) {-} (k{-}1)r_n\bigr]{-} 44(r_n {+} 1)(k {-} 1)(k{-}4) \nonumber \\
= & -(14r_n {+} 44)(k {-} 6)^2 {-} (15r_n^3 {-} 141r_n^2 {+} 323r_n {-} 70)(k {-} 5)
{-} (75r_n^3 {-} 573r_n^2 {+} 891r_n {-} 1362) \nonumber \\
< & ~0, \label{eq:lemma54-case3-numerator}
\end{align}
where the last inequality in \eqref{eq:lemma54-case3-numerator}
follows directly from the expansions
\begin{align*}
15r_n^3-141r_n^2+323r_n-70
& =15\Big(r_n-\frac{13}{2}\Big)^3 +\frac{303}{2}\Big(r_n-\frac{13}{2}\Big)^2 \\
& \quad +\frac{1565}{4}\Big(r_n-\frac{13}{2}\Big)+\frac{1533}{8}>0,
\end{align*}
and
\begin{align*}
75r_n^3-573r_n^2+891r_n-1362
& =75\Big(r_n-\frac{13}{2}\Big)^3 +\frac{1779}{2}\Big(r_n-\frac{13}{2}\Big)^2 \\
& \quad +\frac{11793}{4}\Big(r_n-\frac{13}{2}\Big) + \frac{6537}{8} > 0.
\end{align*}
In the last expression of \eqref{eq:lemma54-case3-numerator},
the square term is nonpositive, the term containing $k-5$ is
nonpositive because $k\geq5$, and the remaining term is strictly
negative. Thus, the entire numerator is strictly negative.
Since its denominator is positive, \eqref{eq:lemma54-case3-bound}
proves the desired inequality for $k\geq5$.
\end{proof}
\begin{proposition}\label{prop:two-universal}
The extremal graph $G$ has two universal vertices.
\end{proposition}

\begin{proof}
By Proposition~\ref{prop:universal}, $G$ has
a universal vertex $v_1$. Suppose, for a contradiction, that it is the only universal vertex. 

We first bound $E$ from above. Recall that $c_v = \max\{0,\, d(v) - 4\}$ and $E = \sum_{v} c_vx_v$. We have
\begin{equation*}
\sum_{v\neq v_1} c_v = 2n-12+k - (n-5) = n + k - 7.
\end{equation*}
Since $v_1$ is the unique universal vertex, every $v\neq v_1$ has degree at most $n-2$. Therefore,
\begin{align}
E & = (n-5) x_1 + yc_{v_2} + \sum_{v\notin\{v_1,v_2\}} c_vx_v 
\leq \frac{n-5}{\lambda + 1} + (y-z) c_{v_2} + z\sum_{v\neq v_1} c_v \nonumber \\
& \leq \frac{n-5}{\lambda + 1} + (n-6)(y-z) + z(n + k - 7) \nonumber \\
& < \frac{n - 5}{r_n + 1} + (n - 6)y + (k - 1)z. \label{eq:sum-v-v1}
\end{align} 
Combining \eqref{eq:lambda-E-L}, \eqref{eq:sum-v-v1}, and the fact that $L > L_0(r_n)$, we have 
$$\lambda - 4 = E - L \leq \frac{n - 5}{r_n + 1} + (n - 6)y + (k - 1)z-L_0(r_n).$$

On the other hand, Lemma \ref{lem:normalized-second} gives
\[
\frac{n-5}{r_n + 1} + (n - 6)y + (k-1) z - L_0(r_n) < r_n - 4.
\]
Combining these two inequalities, we have $\lambda < r_n$,
a contradiction. Thus $G$ has two
universal vertices.
\end{proof}

\section{Completion of the proof}
\label{sec:finish-proof}

Proposition \ref{prop:two-universal} shows that, 
for $n\geq 15$, every $n$-vertex planar graph with maximum spectral radius contains two
vertices of degree $n-1$. In this section, we use this structural
information to complete the proof of our main theorem.

\begin{proof}[Proof of Theorem \ref{thm:main}]
For $n=3$, the unique extremal graph is $K_3=H_3$.

For $4 \leq n\leq 14$, we determine the extremal graphs by exhaustive enumeration of maximal planar graphs and exact comparison of their spectral radii using SageMath (see the \href{https://github.com/ahhylau/Codes_for_max_planar_graphs_spectral_radius}{code}). In the following we assume $n\geq 15$.
Let $v_1,v_2$ be the two universal vertices supplied by Proposition
\ref{prop:two-universal}, and put $G':= G - \{v_1,v_2\}$.
If a vertex $w\in V(G')$ had three neighbors $a,b,c$ in $G'$, then
the sets $\{v_1,v_2,w\}$ and $\{a,b,c\}$ would span a $K_{3,3}$. Hence $\Delta(G')\leq 2$.

The graph $G'$ contains no cycle. Indeed, if it contained a cycle
$C_{\ell}$, then the subgraph on that cycle together with $v_1,v_2$ would
contain $K_2\vee C_{\ell}$. This graph has
$1 + 2\ell + \ell = 3\ell + 1$
edges on $\ell + 2$ vertices, exceeding the planar bound
$3(\ell + 2) - 6 = 3\ell$. Thus $G'$ is acyclic. Together with $\Delta(G')\leq 2$, this
shows that $G'$ is a disjoint union of paths and isolated vertices.

Finally, since $G$ is a triangulation,
\begin{align*}
e(G')
& = e(G) - (1 + 2(n-2)) = 3n - 6 - (2n - 3) = |V(G')| - 1.
\end{align*}
An acyclic graph on $|V(G')|$ vertices with $|V(G')|-1$ edges is
connected. Therefore $G' = P_{n-2}$, and
$G=K_2\vee P_{n-2}$.
This proves Theorem \ref{thm:main}.
\end{proof}

\section*{Declaration on the use of AI}

The authors used language-model-based tools to generate code for searching the extremal
planar graphs for $n\leq 14$ and to assist with several computations
and symbolic derivations. These tools were also used for grammar checking,
language polishing, and improving the clarity of the exposition.
The core ideas and overall strategy of the proof were developed by the
authors. The authors take full responsibility for the correctness and
originality of the paper.

\end{document}